\documentclass[a4paper, 11pt,]{article}
\usepackage{amsfonts}
\usepackage{mathrsfs}

\usepackage{amsmath,amsthm}
\usepackage{amsfonts,amssymb}
\usepackage{multicol}
\def\normo#1{\left\|#1\right\|}
\def\abs#1{|#1|}

\def\brk#1{\left(#1\right)}

\def\rev#1{\frac{1}{#1}}
\def\half#1{\frac{#1}{2}}
\def\norm#1{\|#1\|}

\newcommand{\N}{{\mathbb N}}
\newcommand{\A}{{\mathcal A}}

\newcommand{\R}{{\mathbb R}}
\newcommand{\C}{{\mathbb C}}
\newcommand{\Z}{{\mathbb Z}}

\newcommand{\les}{{\lesssim}}

\newtheorem{theorem}{Theorem}

\newtheorem{proposition}{Proposition}
\newtheorem{definition}{Definition}
\newtheorem{lemma}{Lemma}
\newtheorem{corollary}{Corollary}

\newtheorem{remark}{Remark}

\begin{document}
\title{{\bf Decay estimates for a class of wave equations}}

\author{\bf Zihua Guo, Lizhong Peng, Baoxiang Wang \date{}\\
{\small \it
LMAM, School of Mathematical Sciences, Peking University, Beijing 100871, China}\\
{\small E-mails: zihuaguo, lzpeng, wbx@math.pku.edu.cn} }\maketitle


{\bf Abstract}\quad In this paper we use a unified way studying  the
decay estimate for a class of dispersive semigroup given by
$e^{it\phi(\sqrt{-\Delta})}$, where $\phi: \mathbb{R}^+\rightarrow
\mathbb{R}$ is  smooth away from the origin. Especially, the decay
estimates for the solutions of the Klein-Gordon equation and the
beam equation are simplified and slightly improved.

{\bf Keywords:}\quad Decay estimates, dispersive wave equations

{\bf 2000 MS Classification:} \ 42B25, \ 35F20

\section{Introduction}

In this paper, we study the decay estimate for a class of dispersive
equations:
\begin{equation}\label{e1}
i\partial_tu=-\phi(\sqrt{-\Delta})u+f,\ \ u(0)=u_0(x),
\end{equation}
where $\phi: \mathbb{R}^+\rightarrow \mathbb{R}$ is smooth,
$f(x,t),\ u(x,t):\R^n\times \R\rightarrow\C,\ n\geq 1$, and
$\phi(\sqrt{-\Delta})u=\mathscr{F}^{-1}\phi(|\xi|)\mathscr{F}u$.
Here $\mathscr{F}$ denotes Fourier transform.

Many dispersive wave equations reduce to this type, for instance,
the Schr\"odinger equation ($\phi(r)=r^2$), the wave equation
($\phi(r)=r$), the Klein-Gordon equation ($\phi(r)=\sqrt{1+r^2}$)
and the beam equation ($\phi(r)=\sqrt{1+r^4}$). In 1977,  Strichartz
\cite{STR} derived the priori estimates of the solution to
\eqref{e1} in space-time norm by the Fourier restriction theorem of
Stein and Tomas. Later, his results was improved via a dispersive
estimate and duality argument (cf. \cite{KT},\cite{CW} and
references therein). The dispersive estimate
\begin{equation}\label{e2}
\norm{e^{it\phi(\sqrt{-\Delta})}u_0}_{X} \lesssim
|t|^{-\theta}\norm{u_0}_{X'}
\end{equation}
plays a crucial role, where $X'$ is the dual space of $X$. Applying
\eqref{e2}, together with a standard argument (\cite{KT},
\cite{Wang}), we can immediately get the Strichartz estimates.  When
$\phi$ is a homogenous function of order $m$, namely, $\phi(\lambda
r)=\lambda^m \phi(r)$ for $ \lambda >0$, one can easily obtain a
dispersive estimate (\ref{e2}) by a theorem of Littman and dyadic
decomposition, which is related to the rank of the Hessian
$(\frac{\partial^2 \phi}{\partial \xi_i \xi_j})$. This technique
also works very well when $\phi$ is not radial (\cite{Pecher}).
However, this issue becomes very complicated when $\phi$ is not
homogenous, the main reason is that the scaling constants can not be
effectively separated from the time, cf. Brenner \cite{Brenner},
Lavandosky \cite{LEV}. In this paper, we overcome this difficulty
via frequency localization by separating $\phi$ between high and low
frequency. In higher spatial dimensions, since $\phi$ is radial, we
can reduce the problem to an oscillatory integral in one dimension
by using the Bessel function. Using the dyadic decomposition and
some properties of the Bessel function, we can derive a decay
estimate, as desired. Some earlier ideas on this technique can be
found in \cite{BKS}, \cite{Brenner}, \cite{LEV}.

Since $\phi$ is not homogenous, our idea is to treat the high
frequency and the low frequency in different scales. We will assume
$\phi: \mathbb{R}^+\rightarrow \mathbb{R}$ is smooth and satisfies

(H1) There exists $m_1>0$, such that for any $\alpha \geq 2$ and
$\alpha \in \N$,
\begin{eqnarray*}
|\phi'(r)|\sim r^{m_1-1}\ and\ |\phi^{(\alpha)}(r)|\lesssim
r^{m_1-\alpha},\ \  \ r\geq 1.
\end{eqnarray*}

(H2)  There exists  $m_2>0$, such that for any $\alpha \geq 2$ and
$\alpha \in \N$,
\begin{eqnarray*}
|\phi'(r)|\sim r^{m_2-1}\ and\ |\phi^{(\alpha)}(r)|\lesssim
r^{m_2-\alpha},\ \ 0<r<1.
\end{eqnarray*}

(H3) There exists  $ \alpha_1$, such that
\begin{eqnarray*}
|\phi''(r)|\sim r^{\alpha_1-2}\ \ r\geq 1.
\end{eqnarray*}

(H4) there exists  $\alpha_2$, such that
\begin{eqnarray*}
|\phi''(r)|\sim r^{\alpha_2-2}\ \ 0<r<1.
\end{eqnarray*}

\begin{remark} \rm
(H1) and (H3) reflect the homogeneous order of $\phi$ in high
frequency. If $\phi$ satisfies (H1) and (H3), then $\alpha_1\leq
m_1$. Similarly, the homogeneous order of $\phi$ in low frequency is
described by (H2) and (H4). If $\phi$ satisfies (H2) and (H4), then
$\alpha_2\geq m_2$. The special case $\alpha_2= m_2$ happens in the
most of time.
\end{remark}

Let $\Phi(x): \mathbb{R}^n\to [0,1]$ be a even, smooth radial
function such that \mbox{supp}$\Phi \subseteq \{x:\abs{x}\leq 2\}$,
and $\Phi(x)=1$, if $\abs{x}\leq 1$. Let $\psi(x)=\Phi(x)-\Phi(2x)$,
and $\triangle_k$ be the Littlewood-Paley projector, namely
$\triangle_k f=\mathscr{F}^{-1}\psi(2^{-k} \xi)\mathscr{F}f$. Let
$P_{\leq 0} f=\mathscr{F}^{-1}\Phi(\xi)\mathscr{F} f$. Now we state
our main result:

\begin{theorem}\label{t1}
Assume $\phi: \mathbb{R}^+\rightarrow \mathbb{R}$ is smooth away
from origin, We have the following results.\\
(a) For $k \geq 0$, $\phi$ satisfies (H1), then
\begin{equation}\label{e3}
\norm{e^{it\phi(\sqrt{-\Delta})}\triangle_k u_0}_\infty \lesssim
|t|^{-\theta}2^{k(n-m_1\theta)}\norm{u_0}_1,\ 0\leq \theta \leq
\frac{n-1}{2}.
\end{equation}
In addition, if  $\phi$ satisfies (H3), then
\begin{equation}\label{e3-1}
\norm{e^{it\phi(\sqrt{-\Delta})}\triangle_k u_0}_\infty \lesssim
|t|^{-\half{n-1+\theta}}2^{k(n-\half{m_1(n-1+\theta)}-\half{\theta(\alpha_1-m_1)})}\norm{u_0}_1,\
0\leq \theta \leq 1.
\end{equation}
(b) For $k < 0$, $\phi$ satisfies (H2), then
\begin{eqnarray}\label{e4}
\norm{e^{it\phi(\sqrt{-\Delta})}\triangle_k u_0}_\infty \lesssim
|t|^{-\theta}2^{k(n-m_2\theta)}\norm{u_0}_1,\ 0\leq \theta \leq
\frac{n-1}{2}.
\end{eqnarray}
In addition, if  $\phi$ satisfies (H4), then
\begin{equation}\label{e4-1}
\norm{e^{it\phi(\sqrt{-\Delta})}\triangle_k u_0}_\infty \lesssim
|t|^{-\half{n-1+\theta}}2^{k(n-\half{m_2(n-1+\theta)}-\half{\theta(\alpha_2-m_2)})}\norm{u_0}_1,\
0\leq \theta \leq 1.
\end{equation}
(c) If $\phi$ satisfies (H2), then
\begin{equation}\label{e5}
\norm{e^{it\phi(\sqrt{-\Delta})}P_{\leq0} u_0}_\infty \lesssim
(1+|t|)^{-\theta}\norm{u_0}_1,\  \theta = \min \left(\frac{n}{m_2},
\half{n-1}\right).
\end{equation}
In addition,  if(H4) holds and $\alpha_2=m_2$, then
\begin{eqnarray} \label{e6}
\norm{e^{it\phi(\sqrt{-\Delta})}P_{\leq0} u_0}_\infty \lesssim
(1+|t|)^{-\theta}\norm{u_0}_1,\  \theta = \min \left(\frac{n}{m_2},
\half{n} \right).
\end{eqnarray}
\end{theorem}

\begin{remark} \rm
If $m_1=\alpha_1$ and (H1) and (H3) hold, one can easily verify that
for $k>0$,  \eqref{e3} and \eqref{e3-1} are equivalent to
\begin{equation*}
\norm{e^{it\phi(\sqrt{-\Delta})}\triangle_k u_0}_\infty \lesssim
|t|^{-\theta}2^{k(n-m_1\theta)}\norm{u_0}_1,\ 0\leq \theta \leq
\frac{n}{2}.
\end{equation*}
If $m_2=\alpha_2$ and (H2) and (H4) hold,  one can easily verify
that for $k\le 0$, \eqref{e4} and \eqref{e4-1} are equivalent to
\begin{equation*}
\norm{e^{it\phi(\sqrt{-\Delta})}\triangle_k u_0}_\infty \lesssim
|t|^{-\theta}2^{k(n-m_2\theta)}\norm{u_0}_1,\ 0\leq \theta \leq
\frac{n}{2}.
\end{equation*}

\end{remark}

Throughout this paper, $C>1$ and $c<1$ will denote positive
universal constants, which can be different at different places.
$A\lesssim B$ means that $A\leq CB$,  and $A \sim B$ stands for
$A\lesssim B$ and $B\lesssim A$. We denote by $p'$ the dual number
of $p \in [1,\infty]$, i.e., $1/p+1/p'=1$. We will use Lebesgue
spaces $L^p:=L^p(\mathbb{R}^n)$, $\|\cdot\|_p :=\|\cdot\|_{L^p}$,
Sobolev spaces $H^{s}_p=(I-\Delta)^{-s/2}L^p$, $H^s:=H^s_2$.   Let
$1\le p,q \le \infty. $ Besov spaces are defined in the following
way:
$$B^s_{p, q}=\bigg \{f \in {\cal S}'({\Bbb R}^n):
\| f \|_{B^s_{p, q}}:=\bigg(\|P_{\leq
0}f\|^q_{L^p}+\sum^\infty_{k=1} 2^{k s q}\|\triangle_k f
\|^q_{L^p}\bigg)^{1/q} <\infty\bigg \}.
$$
Some properties of these function spaces can be found in
\cite{BL,Tr}.

 The rest of this paper is organized as follows. In Section 2, we
prove Theorem \ref{t1}. In Section 3, we derive Strichartz estimate
in a general setting. Some applications will be given in Sections 4
and 5.

\section{Decay Estimate } \label{decay}
In this section we will prove Theorem \ref{t1}. The proof for the
case $n=1$ is direct and simple, but reflects the idea for the
higher dimension.
\begin{proof}[\textbf{Proof of Theorem \ref{t1}}]  Since the proof in the
case $n=1$ is slightly different from the case of higher spatial
dimensions, we divide the proof into the following two steps.

\textbf{Step 1.} We consider the case $n=1$.  First, we prove (a).
It follows from Young's inequality that
\begin{eqnarray*}
\norm{e^{it\phi(\sqrt{-\Delta})}\triangle_k u_0}_\infty  &\lesssim&
\norm{\mathscr{F}^{-1}e^{it\phi(|\xi|)}\psi(2^{-k}|\xi|)}_\infty
\norm{u_0}_1\\
&\lesssim& \norm{I_k }_{\infty}\norm{u_0}_1,
\end{eqnarray*}
where we assume that
\begin{eqnarray} \label{e6-1}
I_k(x)=\mathscr{F}^{-1}(e^{it\phi(|\xi|)}\psi(2^{-k}|\xi|))(2^{-k}x)=2^k\int
e^{ix\xi}e^{it\phi(2^k\xi)}\psi(\xi)d\xi.
\end{eqnarray}
We immediately get that
\begin{eqnarray}\label{eq:e21}
\norm{I_k}_{\infty}\lesssim 2^k,
\end{eqnarray}
which is the result of \eqref{e3}, as desired. Now we assume (H3)
holds. Let $\phi_1(\xi)=x\xi+t\phi(2^k\xi)$, then
$|\phi_1''(\xi)|>|t|2^{k\alpha_1}$ on the support of $\psi$. Thus by
van der Corput's Lemma (see \cite{Stein}) we can get
\begin{eqnarray} \label{eq:e21-1}
\norm{I_k}_{\infty}\lesssim |t|^{-\half 1}2^{k(1-\half{\alpha_1})}.
\end{eqnarray}
By an interpolation between  \eqref{eq:e21} and \eqref{eq:e21-1}, we
get for $0\leq \theta \leq 1$,
\begin{eqnarray*}
\norm{I_k}_{\infty}\lesssim |t|^{-\half
\theta}2^{k(1-\half{\theta\alpha_1})},
\end{eqnarray*}
which completes the proof of (a) in the case $n=1$.

The proof of (b) is similar to (a) and we omit the details. Now we
turn to the proof of (c). First, we consider the case $m_2<2$. Fix
$0\leq \theta \leq \min(\rev {m_2},\rev 2)=\rev {2}$. Since
$\theta<\frac{1}{m_2}$,  it follows from (b) that
\begin{eqnarray}\label{e16-a}
\norm{e^{it\phi(\sqrt{-\Delta})}P_{\leq0} u_0}_\infty &\lesssim&
\sum_{k<0}|t|^{-\theta}2^{k(1-\theta m_2)}\norm{u_0}_1
\lesssim|t|^{-\theta}\norm{u_0}_1.
\end{eqnarray}
Taking $\theta=0$ or $\theta= \min(\rev {m_2},\rev 2)$, we get the
result.

Next, we consider the case $m_2\geq 2$. One easily sees that
\eqref{e16-a} holds also in the case $m_2\geq 2$ and $\theta=0$. So,
it suffices to consider the case $m_2\geq 2$ and $\theta=\min(\rev
{m_2},\rev 2)=\rev{m_2}$. By simple calculation, we get that for any
$ m\geq 0$,
\begin{equation}\label{e16}
\frac{d^m}{d{\xi^m}}\brk{\rev{\phi'(2^k\xi)}} \lesssim
2^{-k(m_2-1)}, \quad \xi\in {\rm supp}\psi.
\end{equation}
Thus,  if $|x|\leq 1$, then
$|\partial_{\xi}^m(e^{ix\xi}\psi(\xi))|\lesssim 1$, and integrating
by part we can get that for any $q \geq 0$,
\[
|I_k(x)|\lesssim |t|^{-q}2^{k(1-m_2q)}.
\]
If $|x|>1$, let $k_0$ be the smallest integer such that $|x|\leq
|t|2^{k_0m_2}$, then $|x|\approx |t|2^{k_0m_2}$. For $|k-k_0|>C\gg
1$, one has that $|\phi_1'(\xi)|\geq c|t|2^{km_2}$, integrating by
part we can get that for any $\ q \geq 0$ (see \eqref{e20} below),
\[
|I_k(x)|\lesssim |t|^{-q}2^{k(1-m_2q)}.
\]
For $|k-k_0|\leq C$, noticing that $|x|>1$ and $m_2\ge 2$,  we have
\begin{eqnarray*}
|I_k(x)|&\lesssim& |t|^{-\half
1}2^{k(1-\half{m_2})}\lesssim|t|^{-\half
1}\brk{\frac{|x|}{|t|}}^{(1-\half{m_2})\rev{m_2}} \lesssim
|t|^{-\rev{m_2}}.
\end{eqnarray*}
Therefore, taking $q$ sufficiently large, we have
\begin{eqnarray*}
|\sum_{k\le 0} I_{k}(x)| &\lesssim& \sum_{|k-k_0|\leq C} |I_k(x)|+ \sum_{|k-k_0|\geq C} |I_k(x)| \\
&\lesssim& \sum_{|k-k_0|\leq
C}|t|^{-\rev{m_2}}+\sum_{2^k<|t|^{-\rev{m_2}}}2^k+\sum_{2^k>|t|^{-\rev{m_2}}}|t|^{-q}2^{k(1-m_2q)}\\
&\les&|t|^{-\rev{m_2}},
\end{eqnarray*}
which completes the proof of (c).

\textbf{Step 2.} We consider the case $n\geq 2$. Our idea is as
follows: First, we reduce the problem to an oscillatory integral in
one dimension relating the Bessel function by changing to polar
coordinates; Next, we divide the discussion into two cases: in one
case we use the vanishing property at the origin and the recurring
property for the Bessel function, and in another case we use the
decay property of the Bessel function. We denote by $J_m(r)$ the
Bessel function:
\begin{eqnarray*}
J_m(r)=\frac{(r/2)^m}{\Gamma(m+1/2)\pi^{1/2}}
\int_{-1}^1e^{irt}(1-t^2)^{m-1/2}dt, \ \ m>-1/2.
\end{eqnarray*}
 We first list some properties
of $J_m(r)$ in the following lemma. For their proof we refer the
readers to \cite{Stein}, \cite{GRA}.
\begin{lemma}[Properties of the Bessel function]\label{l1}
 We have for $0<r<\infty$ and $m>-\half 1$\\
{\rm (i)} $J_m(r)\leq Cr^m$,\\
{\rm (ii)} $\frac{d}{dr}(r^{-m}J_m(r))=-r^{-m}J_{m+1}(r)$,\\
{\rm (iii)} $J_m(r)\leq Cr^{-\half 1}$.
\end{lemma}
It is well known that the Fourier transform of a radial function $f$
is still radial and (cf. \cite{Stein2})
\begin{eqnarray}\label{e8}
\hat{f} (\xi)= 2\pi  \int_{0}^{\infty} f(r) r^{n-1}(r
|\xi|)^{-\half{n-2}}J_{\frac{n-2}{2}}(r|\xi|)dr,
\end{eqnarray}
From (i) and (ii) of Lemma \ref{l1}, we can easily get that for
$0\leq s \leq 2$ and for any $\ k\geq 0$,
\begin{equation}\label{e15}
\left |\frac{\partial^k}{\partial
r^k}\brk{\psi(r)r^{n-1}(rs)^{-(n-2)/2}J_{\frac{n-2}{2}}(rs)}\right
|\leq c_k.
\end{equation}
If $m=-\frac{n-2}{2}$, $J_m(r)$ is connected to the Fourier
transform of the spherical surface measure. It is known that (see
\cite{John}, Ch. 1, Equation (1.5)),
\begin{equation}\label{e17}
r^{-\frac{n-2}{2}}J_{\frac{n-2}{2}}(r)=c_n\mathscr{R}(e^{ir}h(r)),
\end{equation}
where $h$ satisfies
\begin{equation}\label{e12}
|\partial_r^kh(r)|\leq c_k (1+r)^{-\frac{n-1}{2}-k}.
\end{equation}
From (\ref{e12}), we get that, for $s>2$ and for any $k\geq 0$,
\begin{equation}\label{e18}
\left |\frac{\partial^k}{\partial
r^k}\brk{\psi(r)r^{n-1}h(rs)}\right |\leq c_k s^{-\frac{n-1}{2}}.
\end{equation}

We now show the proof of (a). It follows from Young's inequality
that
\begin{eqnarray*}\label{e6}
\norm{e^{it\phi(\sqrt{-\Delta})}\triangle_k
u_0}_\infty&=&\norm{\mathscr{F}^{-1}e^{it\phi(|\xi|)}\psi(2^{-k}|\xi|)\mathscr{F}u_0}_\infty
\\
&\lesssim&
\norm{\mathscr{F}^{-1}e^{it\phi(|\xi|)}\psi(2^{-k}|\xi|)}_\infty
\norm{u_0}_1.
\end{eqnarray*}
In view of \eqref{e8} we have
\begin{align*}\label{e6}
\mathscr{F}^{-1}(e^{it\phi(|\xi|)}\psi(2^{-k}|\xi|))(x)=& 2^{kn}\mathscr{F}^{-1}(e^{it\phi(|2^k\xi|)}\psi(|\xi|))(2^k|x|)\\
=&2^{kn}\int_{0}^{\infty}e^{it\phi(2^kr)}\psi(r)r^{n-1}(r2^ks)^{-\half{n-2}}J_{\frac{n-2}{2}}(r2^ks)dr\\
:= & II_k(2^ks),
\end{align*}
where $s=|x|$. It suffices to show
\begin{eqnarray*}
\norm{II_k(s)}_\infty \les |t|^{-\theta}2^{k(n-m_1\theta)}.
\end{eqnarray*}
 From (i) of Lemma \ref{l1}, we
obtain the trivial estimate for $\theta=0$,
\begin{equation}\label{e7}
\norm{II_k(s)}_\infty\lesssim 2^{kn}.
\end{equation}
We will discuss it in following two cases.

\textbf{Case 1.} $s\leq 2$. In this case, we will use the vanishing
property of the Bessel function at the origin. Denote
$D_r=\rev{it\phi'(2^kr)2^k}\frac{d}{dr}$. We see that
$$
D_r(e^{it\phi(2^kr)})=e^{it\phi(2^kr)},\
(D_r)^*f=-\rev{it2^k}\frac{d}{dr}\left(\rev{\phi'(2^kr)}f\right).
$$
From (H1), we get that for any $m\geq 0$ and $r\sim 1$,
\begin{equation}\label{e16}
\frac{d^m}{dr^m}\brk{\rev{\phi'(2^kr)}}\leq c_m 2^{-k(m_1-1)}.
\end{equation}
Let $\tilde{\psi}(r)=\psi(r)r^{n-1}$.  Using integration by part, we
have for any $q\in \Z^{+}$,
\begin{eqnarray}
II_k(s)&=&2^{kn}\int_{0}^{\infty}e^{it\phi(2^kr)}\tilde{\psi}(r)(rs)^{-\frac{n-2}{2}}J_{\frac{n-2}{2}}(rs)dr \nonumber\\
&=&2^{kn}\int_{0}^{\infty}D_r(e^{it\phi(2^kr)})\tilde{\psi}(r)(rs)^{-\frac{n-2}{2}}J_{\frac{n-2}{2}}(rs)dr \nonumber\\
&=&-\frac{2^{kn}}{it2^k}\int_{0}^{\infty}e^{it\phi(2^kr)}\frac{d}{dr}
\brk{\rev{\phi'(2^kr)}\tilde{\psi}(r)(rs)^{-\frac{n-2}{2}}J_{\frac{n-2}{2}}(rs)}dr \nonumber\\
&=&\frac{2^{kn}}{(it2^k)^q}\sum_{m=0}^q \sum_{l_1,\ldots l_q \in
\Lambda_m^q}C_{q,m} \nonumber\\
& &\cdot\int_{0}^{\infty}e^{it\phi(2^kr)}
\prod_{j=1}^{q}\partial_r^{l_j}\brk{\rev{\phi'(2^kr)}}
\partial_r^{q-m}\brk{\tilde{\psi}(r)(rs)^{-\frac{n-2}{2}}J_{\frac{n-2}{2}}(rs)}dr,
\label{e20}
\end{eqnarray}
where $\Lambda_m^q=\{l_1,\ldots,l_q \in \Z^{+}: 0\leq l_1<\ldots
<l_q\leq q, l_1+\ldots l_q=m\}$. It follows from (\ref{e15}),
(\ref{e16}) and \eqref{e20} that, for any $q\in \Z^{+}$,
\begin{align}
|II_k(s)|\lesssim |t|^{-q}2^{k(n-m_1q)}. \label{e22}
\end{align}
Interpolating \eqref{e22} with (\ref{e7}), we get that for any
$\theta \geq 0$, $|I_k(s)|\lesssim |t|^{-\theta}2^{k(n-m_1\theta)},
$ which completes the proof of (a) in this case.

\textbf{Case 2.} $s\geq 2$. In this case, we will use the decay
property of Bessel function. It follows from (\ref{e17}) that
\begin{eqnarray*}
II_k(s)&=&c_n2^{kn}\int_{0}^{\infty}e^{it\phi(2^kr)}\tilde{\psi}(r)(e^{irs}h(rs)+e^{-irs}\bar{h}(rs))dr\\
&=&c_n2^{kn}\int_{0}^{\infty}e^{i(t\phi(2^kr)+rs)}\tilde{\psi}(r)h(rs)dr+c_n2^{kn}\int_{0}^{\infty}e^{i(t\phi(2^kr)-rs)}\tilde{\psi}(r)\bar{h}(rs)dr \\
&:= &B_1+B_2.
\end{eqnarray*}
Without loss of generality, we can assume that $t>0$ and
$\phi'(r)>0$. For $B_1$, let $\phi_1(r)=t\phi(2^kr)+rs$. Note that
$\phi_1'(r)=t2^k\phi'(2^kr)+s\geq c t 2^{km_1}$, and (\ref{e16})
also holds if we replace $\phi$ by $\phi_1$. Noticing (\ref{e18}),
analogous to Case 1 we can get that for any $\theta \geq 0$,
\begin{eqnarray*}
|B_1|\lesssim |t|^{-\theta}2^{k(n-m_1\theta)}.
\end{eqnarray*}
For $B_2$, let $\phi_2(r)=t\phi(2^kr)-rs$. We note that if
$s=t2^k\phi'(2^kr)$, then $\phi_2'(r)=0$. We divide the discussion
into the following two cases.

\textbf{Case 2a.} $s>2\sup_{r \in [1/2,2]}t2^k\phi'(2^kr)$ or
$s<\half 1\inf_{r \in [1/2,2]}t2^k\phi'(2^kr)$. In this case, we see
that $|\phi_2'(r)|\geq c t2^{km_1}$ if $r\sim 1$, and (\ref{e16})
also holds if one replaces $\phi$ by  $\phi_2$. By (\ref{e18}), we
can get that for any $\theta \geq 0$,
\begin{eqnarray*}
|B_2|\lesssim |t|^{-\theta}2^{k(n-m_1\theta)}.
\end{eqnarray*}

\textbf{Case 2b.} $\half 1\inf_{r \in [1/2,2]}t2^k\phi'(2^kr)\leq
s\leq 2\sup_{r \in [1/2,2]}t2^k\phi'(2^kr)$. It follows from
(\ref{e18}) that
\begin{eqnarray}
|B_2|\lesssim 2^{kn} s^{-\frac{n-1}{2}}\lesssim
t^{-\frac{n-1}{2}}2^{k(n-\half{(n-1)m_1})}. \label{e21}
\end{eqnarray}
Interpolating \eqref{e21} with (\ref{e7}), we get that for $0\leq
\theta \leq \frac{n-1}{2}$,
\begin{equation}\label{e19}
|B_2|\les t^{-\theta}2^{k(n-m_1\theta)}.
\end{equation}

If  (H3) holds in addition, then $|\phi_2''(r)|\geq t2^{k\alpha_1}.$
It follows from van der Corput's Lemma that
\begin{eqnarray}
|B_2|&\lesssim& (t2^{k\alpha_1})^{-1/2}\int_0^\infty
|\frac{d}{dr}(\tilde{\psi(r)}h(rs))|dr \lesssim
t^{-n/2}2^{k(n-\half{n}(m_1+\frac{\alpha_1-m_1}{n}))}. \label{e24}
\end{eqnarray}
Therefore, interpolating \eqref{e24} with (\ref{e19}) and using the
fact that for $0\leq \theta \leq 1$,
$\half{n-1+\theta}=(1-\theta)\half{n-1}+\theta\half{n}$, we get
\begin{eqnarray*}
|II_k(s)|\lesssim
|t|^{-\half{n-1+\theta}}2^{k(n-\half{m_1(n-1+\theta)}-\half{\theta(\alpha_1-m_1)})},
\quad 0\leq \theta \leq 1,
\end{eqnarray*}
which completes the proof of (a).

The proof of (b) is similar to that of (a) and we omit the details.
Now we turn to proof of (c). Fix $0\leq \theta \leq
\min(\frac{n}{m_2},\half{n-1})$, If $\theta <\frac{n}{m_2}$, then
$n-m_2\theta>0$. From (b), we have
\begin{eqnarray*}
\norm{e^{it\phi(\sqrt{-\Delta})}P_{\leq0} u_0}_\infty &\lesssim&
\sum_{k=-\infty}^2 \norm{e^{it\phi(\sqrt{-\Delta})}\triangle_k
P_{\leq0}
u_0}_\infty\\
&\lesssim&\sum_{k=-\infty}^2
|t|^{-\theta}2^{k(n-m_2\theta)}\norm{P_{\leq0}
u_0}_1\\
&\lesssim& |t|^{-\theta}\norm{P_{\leq0} u_0}_1.
\end{eqnarray*}
Now we assume $\half{n-1}\geq \frac{n}{m_2}$ and $
\theta=\frac{n}{m_2}$ in the following discussion. From the proof of
(b), we know that, if $k_0<0$ and $s\sim t2^{k_0m_2}\geq 2$, then
\begin{eqnarray*}
|II_{k_0}(s)|&\lesssim& t^{-\half{n-1}}2^{k_0(n-\half{(n-1)m_2})}\\
&\lesssim&t^{-\half{n-1}}(\frac{s}{t})^{(n-\half{(n-1)m_2})\rev{m_2}}\\
&\lesssim&t^{-\frac{n}{m_2}}.
\end{eqnarray*}
If $|k-k_0|>C\gg 1$, then
\begin{eqnarray*}
|II_k(s)|\lesssim t^{-\alpha}2^{k(n-m_2\alpha)},\ \forall\ \alpha \
\geq 0.
\end{eqnarray*}
Therefore,  choosing $\alpha$ large, we have
\begin{eqnarray*}
|II_{\leq0}(s)|&\lesssim& \sum_{|k-k_0|\leq C}|II_{k}(s)|+\sum_{|k-k_0|>C}|II_k(s)|\\
&\lesssim&t^{-\frac{n}{m_2}}+\sum_{2^k<t^{-\rev{m_2}}}2^{kn}+\sum_{2^k>t^{-\rev{m_2}}}t^{-\alpha}2^{k(n-m_2\alpha)}\\
&\lesssim&t^{-\frac{n}{m_2}}.
\end{eqnarray*}
If in addition (H4) holds, the proof is similar. We omit the
details.
\end{proof}
\begin{remark} \rm
It's easy to see that in the case $n=1$ we did not use the
properties that $\phi$ is even. Our method is also adapted to more
general radial $\phi$. But it seems difficult to apply for
non-radial $\phi$.
\end{remark}

\section{Strichartz Estimate}

In this section, we show the Strichartz estimate by using the decay
estimates obtained in Section \ref{decay}. We will work in an
general setting in this section and apply it to some concrete
equation in the next section. Our method is using duality argument.
We mention that this argument is quite standard. We will omit most
of the proof, and refer the reader to \cite{KT} for details. Here we
use an argument in \cite{Wang}. Since the decay rate is different
between $|t|>1$ and $|t|\leq 1$, we will need a variant
Hardy-Littlewood-Sobolev inequality.
\begin{lemma}\label{l2}
Assume $\gamma_1, \gamma_2\in \R$, let
\begin{eqnarray*}
k(y)= \left \{
\begin{array}{l}
|y|^{-\gamma_1},\quad |y|\leq1,\\
|y|^{-\gamma_2}, \quad |y|>1,
\end{array}
\right.
\end{eqnarray*}
Assume that one of the following
conditions holds,\\
(a) $0<\gamma_1=\gamma_2<n$, $1<p<q<\infty$ and $1-\rev p+\rev q=\frac{\gamma_1}{n}$,\\
(b) $\gamma_1<\gamma_2$, $0<\gamma_1<n$, $1<p<q<\infty$ and $1-\rev p+\rev q=\frac{\gamma_1}{n}$,\\
(c) $\gamma_1<\gamma_2$, $0<\gamma_2<n$, $1<p<q<\infty$ and $1-\rev p+\rev q=\frac{\gamma_2}{n}$,\\
(d) $\gamma_1<\gamma_2$, $1\leq p\leq q\leq\infty$ and
$\frac{\gamma_1}{n}<1-\rev p+\rev q<\frac{\gamma_2}{n}$.\\
We have
$$\norm{f*k}_q\les \norm{f}_p.$$
\end{lemma}
\begin{proof} By splitting $\mathbb{R}^n$ into $|y|\ge 1$ and $|y|\le 1$, we can easily get the
results by following Hardy-Littlewood-Sobolev's and  Young's
inequalities.
\end{proof}

\begin{definition} Given $\theta_1\leq \theta_2$, we say $q$ belongs to
$E(\theta_1,\theta_2)$ if one of the following holds: \\
(a) $0<\theta_1=\theta_2<1$ and $q=\frac{2}{\theta_1}$,\\
(b) $\theta_1<\theta_2$, $0<\theta_1<1$ and $q=\frac{2}{\theta_1}$,\\
(c) $\theta_1<\theta_2$, $0<\theta_2<1$ and $q=\frac{2}{\theta_2}$,\\
(d) $\theta_1<\theta_2$, $2\leq q\leq\infty$ and
$\theta_1<\frac{2}{q}<\theta_2$.
\end{definition}

\quad We now give the Strichartz estimate. Denote
\[U(t)=e^{it\phi(\sqrt{-\Delta})},\quad \A
f=\int_0^tU(t-\tau)f(\tau,\cdot)d\tau.\] We assume that, for $2\leq
p\leq \infty$, $\alpha:=\alpha(p)\in \R$, and $\theta_1\leq
\theta_2$,
\begin{equation}\label{eq:e20}
\norm{U(t)f}_{B_{p,2}^{\alpha}}\les k(t)\norm{f}_{B_{p',2}^0},
\end{equation}
where
\begin{eqnarray*}
k(t)= \left \{
\begin{array}{l}
|t|^{-\theta_1},\quad |t|\leq1,\\
|t|^{-\theta_2}, \quad |t|>1.
\end{array}
\right.
\end{eqnarray*}
Using Lemma \ref{l2} and standard duality argument, we can prove the
following proposition. We omit its proof and refer the reader to see
\cite{KT}, \cite{Wang}.
\begin{proposition}\label{p1}
Assume $U(t)$ satisfies \eqref{eq:e20}, then we have for $q \in
E(\theta_1,\theta_2)$, $\eta \in \R$, and $T>0$,
\begin{eqnarray*}
\norm{U(t)h}_{L^q(-T,T;B_{p,2}^{\eta+\half \alpha})}&\les&
\norm{h}_{H^\eta},\\
\norm{\A f}_{L^q(-T,T;B_{p,2}^{\eta+\alpha})}&\les&
\norm{f}_{L^{q'}(-T,T;B_{p',2}^\eta)},\\
\norm{\A f}_{L^{\infty}(-T,T;H^{\eta+\half \alpha})}&\les&
\norm{f}_{L^{q'}(-T,T;B_{p',2}^\eta)},\\
\norm{\A f}_{L^q(-T,T;B_{p,2}^{\eta+\half \alpha})}&\les&
\norm{f}_{L^1(-T,T;H^\eta)}.
\end{eqnarray*}
\end{proposition}
\begin{remark} \rm
The endpoint case $\theta_i=1$ also holds by following Keel and
Tao's ideas in \cite{KT}, but we will not pursue this issue in this
paper.
\end{remark}

\section{Application}
In this section we will apply Theorem \ref{t1} in Section 2 to some
concrete equations.  Our results below can cover some known results
so far, and make some improvements and provide simple proofs for the
Klein--Gordon equation and the Beam equation. A simple case is the
semi-group $e^{it(-\Delta)^\rho}$, $\rho>0$, we do not list its
estimates and one can get the desired estimates by using the same
way as in the following Klein-Gordon equation.

1 (Klein-Gordon equation). First, we consider the Klein-Gordon
equation,
\begin{eqnarray}\label{eq:klein}
\left \{
\begin{array}{l}
\partial_{tt}u-\Delta u+u=F,\\
u(0)=u_0(x),\ u_t(0)=u_1(x).
\end{array}
\right.
\end{eqnarray}
By Duhamel's principle, we get
\begin{eqnarray*}
u=K'(t)u_0+K(t)u_1-\int_0^tK(t-\tau)F(\tau)d\tau,
\end{eqnarray*}
where
\begin{eqnarray*}
K(t)=\omega^{-1} \sin(t\omega),\quad K'(t)=\cos(t\omega), \quad
\omega=\sqrt{I-\Delta}.
\end{eqnarray*}
This reduces to the semigroup $K_{\pm}(t):=e^{\pm
it(I-\Delta)^{1/2}}$, which corresponds to $\phi(r)=(1+r^2)^{1/2}$.
By simple calculation,
\[
\phi'(r)=r/(1+r^2)^{\half 1},\quad \phi''(r)=1/(1+r^2)^{\half 3},
\]
we see that $\phi$ satisfies (H1), (H2), (H3) and (H4) with $m_1=1$,
$\alpha_1=-1$, $m_2=\alpha_2=2$.
\begin{proposition}\label{p2}
Assume $2\leq p\leq \infty$, $1\leq q \leq \infty$, $\rev
p+\rev{p'}=1$, $\delta=\half 1-\rev p$.\\
{\rm (i)} Let $0\leq \theta\leq 1$, and $(n+1+\theta)\delta=
1+s'-s$, we have
\begin{eqnarray*}
\norm{K(t)g}_{B_{p,q}^s}\les \; |t|^{-(n-1+\theta)\delta}
\norm{g}_{B_{p',q}^{s'}}.
\end{eqnarray*}
{\rm (ii)} Let $0\leq \theta\leq n-1$, and $(n+1+\theta)\delta=
1+s'-s$, we have
\begin{eqnarray*}
\norm{K(t)g}_{B_{p,q}^s}\les \; |t|^{-(n-1-\theta)\delta}
\norm{g}_{B_{p',q}^{s'}}.
\end{eqnarray*}
{\rm (iii)} In particular, one has that for $\theta\in [0,1]$,
$(n+1+\theta)\delta\le 1+s'-s$,
\begin{eqnarray*}
\norm{K(t)g}_{B_{p,q}^s}\les k(t)\norm{g}_{B_{p',q}^{s'}}, \quad
k(t)= \left \{
\begin{array}{l}
|t|^{\min(1+s'-s-2n\delta,\
0)},\quad |t|\leq 1,\\
|t|^{-(n-1+\theta)\delta}, \quad |t|\geq 1.
\end{array}
\right.
\end{eqnarray*}
\end{proposition}
\begin{proof} First, we show the results of (i).
It follows from (c) of Theorem \ref{t1} and Plancherel's identity
that,
\begin{align*}
\norm{K_{+}(t)P_{\leq 0}u_0}_\infty & \lesssim
|t|^{-(n-1+\theta)/2}\norm{P_{\leq 0}u_0}_{1}, \\
\norm{K_{+}(t)P_{\leq 0}u_0}_2 & \lesssim  \norm{P_{\leq 0}u_0}_{2}.
\end{align*}
From \eqref{e3-1} of Theorem \ref{t1} and  Plancherel's identity, we
can get for $k>0$,
\begin{align*}
\norm{K_{+}(t)\triangle_ku_0}_\infty & \lesssim
|t|^{-\half{n-1+\theta}}2^{k\half{n+1+\theta}}\norm{\triangle_ku_0}_1,\\
\norm{K_{+}(t)\triangle_k u_0}_2 & \lesssim  \norm{ \triangle_k
u_0}_{2}.
\end{align*}
Thus by Riesz-Thorin theorem,
\begin{align*}
\norm{K_{+}(t)P_{\leq 0}u_0}_p & \lesssim
|t|^{-(n-1+\theta)\delta}\norm{P_{\leq 0}u_0}_{p'}, \\
\norm{K_{+}(t)\triangle_ku_0}_p &\lesssim
|t|^{-(n-1+\theta)\delta}2^{k(n+1+\theta)\delta}\norm{\triangle_ku_0}_{p'}.
\end{align*}
Therefore, it follows from $(n+1+\theta)\delta = 1+s'-s$ that
\begin{eqnarray*}
\norm{K(t)u_0}_{B_{p,q}^{s}}\lesssim |t|^{-(n-1+\theta)\delta}
\norm{u_0}_{B_{p',q}^{s'}},
\end{eqnarray*}
which completes the proof of (i).

Next, we prove (ii). Let us rewrite \eqref{e3} as
\begin{align} \label{e28}
\norm{K_{+}(t)\triangle_ku_0}_\infty & \lesssim
|t|^{-\half{n-1-\theta}}2^{k\half{n+1+\theta}}\norm{\triangle_ku_0}_1,
\quad 0\le \theta\le n-1.
\end{align}
\eqref{e5} implies that
\begin{align} \label{e29}
\norm{K_{+}(t)P_{\le 0} u_0}_\infty & \lesssim
|t|^{-\half{n-1-\theta}} \norm{P_{\le 0} u_0}_1, \quad 0\le
\theta\le n-1.
\end{align}
Hence, using the same way as in the proof of (i), we can easily get
the results of (ii).

Following the proofs above, one can get the results of (iii).
Indeed, it suffices to consider the case $|t|\le 1$,  $\theta \in
[0,1]$ and $1+s'-s-2n\delta\ge 0$. If this case occurs, taking
$\theta=n-1$ in \eqref{e28} and \eqref{e29} one has the result, as
desired.
\end{proof}

2 (Beam equation). We now consider the Beam equation, in some
literature it is called fourth order wave equations,
\begin{eqnarray}\label{eq:Beam}
\left \{
\begin{array}{l}
\partial_{tt}u+\Delta^2 u+u=F,\\
u(0)=u_0(x),\ u_t(0)=u_1(x).
\end{array}
\right.
\end{eqnarray}
By Duhamel's principle, we have
\begin{eqnarray*}
u=B'(t)u_0+B(t)u_1-\int_0^tB(t-\tau)F(\tau)d\tau
\end{eqnarray*}
where
\begin{eqnarray*}
B(t)= \omega^{-1}\sin(t\omega), \quad B'(t)=\cos(t \omega), \quad
\omega=\sqrt{I+\Delta^2}.
\end{eqnarray*}
This reduces to the semigroup $B_{\pm}(t):=e^{\pm
it(I+\Delta^2)^{1/2}}$, which corresponding to
$\phi(r)=(1+r^4)^{1/2}$. By simple calculation,
\[
\phi'(r)=2r^3/(1+r^4)^{\half 1}, \quad
\phi''(r)=(6r^2+2r^6)/(1+r^4)^{\half 3},
\]
we know that $\phi$ satisfies (H1), (H2), (H3) and (H4) with
$m_1=\alpha_1=2$, $m_2=\alpha_2=4$.
\begin{proposition}\label{p3}
Assume $2 \leq p\leq \infty$, $\rev p+\rev{p'}=1$, $1\leq q\leq
\infty$, $\delta=\half 1-\rev p$, and $0\leq 2+s'-s$, then
\begin{eqnarray*}
\norm{B(t)g}_{B_{p,q}^s}\leq k(t)\norm{g}_{B_{p',q}^{s'}},
\end{eqnarray*}
where
\begin{eqnarray*}
k(t)= \left \{
\begin{array}{l}
|t|^{\min(1+\half {s'}-\half s-n\delta,\
0)},\quad |t|\leq 1,\\
|t|^{-\half{n\delta}}, \quad |t|\geq1.
\end{array}
\right.
\end{eqnarray*}
\end{proposition}
\begin{proof} First, we prove the case $|t|\geq 1$. It follows from (c) of
Theorem \ref{t1} by setting $\theta=\frac{n}{4}$ and Riesz-Thorin
interpolation theorem that
\begin{eqnarray*}
\norm{B_{+}(t)P_{\leq 0}u_0}_p \lesssim
|t|^{-\half{n}\delta}\norm{P_{\leq 0}u_0}_{p'}.
\end{eqnarray*}
From (a) of Theorem \ref{t1} by setting $\theta=\half n$ and
interpolation, we get for $k>0$,
\begin{eqnarray*}
\norm{B_{+}(t)\triangle_ku_0}_p \lesssim
|t|^{-n\delta}\norm{\triangle_ku_0}_{p'},
\end{eqnarray*}
and then we have
\begin{eqnarray*}
2^{ks}\norm{(I+\Delta^2)^{-1/2}B_{+}(t)\triangle_ku_0}_p \les \
|t|^{-\half{n}\delta}2^{ks'}\norm{\triangle_ku_0}_{p'}.
\end{eqnarray*}
Therefore,
\[
\norm{B(t)g}_{B_{p,q}^s}\leq
|t|^{-\half{n}\delta}\norm{g}_{B_{p',q}^{s'}},
\]
which completes the proof of the proposition in this case.

For $|t|\leq 1$, from (c) of Theorem \ref{t1} by setting $\theta=0$
and interpolation, we get
\begin{eqnarray*}
\norm{(I+\Delta^2)^{-1/2}B_{+}(t)P_{\leq 0}u_0}_p \lesssim
\norm{P_{\leq 0}u_0}_{p'}.
\end{eqnarray*}
For $k>0$, from (a) of Theorem 1 and interpolation, we get for
$0\leq \theta \leq \half n$,
\begin{eqnarray*}
\norm{B_{+}(t)\triangle_ku_0}_p \lesssim
|t|^{-2\theta\delta}2^{2k(n-2\theta)\delta}\norm{\triangle_ku_0}_{p'},
\end{eqnarray*}
and then that
\begin{eqnarray*}
2^{ks}\norm{(I+\Delta^2)^{-1/2}B_{+}(t)\triangle_ku_0}_p \lesssim
|t|^{-2\theta\delta}2^{-k(2+s'-s-2(n-2\theta)\delta)}2^{ks'}\norm{\triangle_ku_0}_{p'}.
\end{eqnarray*}
If $0\leq 2+s'-s\leq 2n\delta$, then we can choose $0\leq \theta
\leq \half{n}$ such that $2+s'-s-2n\delta=-4\theta\delta$. If
$2+s'-s> 2n\delta$, then we choose $\theta=0$. Thus we get
\begin{eqnarray*}
2^{ks}\norm{(I+\Delta^2)^{-1/2}B_{+}(t)\triangle_ku_0}_p \lesssim
|t|^{\min(1+\half{s'}-\half{s}-n\delta,\
0)}2^{ks'}\norm{\triangle_ku_0}_{p'}.
\end{eqnarray*}
Therefore,
\[
\norm{B(t)g}_{B_{p,q}^s}\les
|t|^{\min(1+\half{s'}-\half{s}-n\delta,\
0)}2^{ks'}\norm{g}_{B_{p',q}^{s'}},
\]
which completes the proof of the proposition.
\end{proof}
\begin{corollary}
\label{levan}  Assume $2\leq q< \infty$, $u(t)$ be the solution of
\eqref{eq:Beam} with $F=0$. Then
\begin{equation}\label{eq:Beam2}
\norm{u(t)}_{L^q(\R^n)}\les
k(t)(\norm{u_0}_{W^{2,q'}}+\norm{u_1}_{L^{q'}}),
\end{equation}
where
\begin{eqnarray*}
k(t)= \left \{
\begin{array}{l}
|t|^{\min(1+\frac{n}{q}-\half n,\ 0)},\quad |t|\leq 1,\\
|t|^{\frac{n}{2q}-\frac{n}{4}}, \quad |t|\geq1.
\end{array}
\right.
\end{eqnarray*}
Furthermore, if $u_0=0$ then for $|t|\leq 1$,
\begin{eqnarray*}
\norm{u(t)}_q\les |t|^{1+\frac{n}{q}-\half n}\norm{u_1}_{q'}.
\end{eqnarray*}
\end{corollary}
\begin{remark} \rm Corollary \ref{levan} is a slightly modified version of
Theorem 2.1 in \cite{LEV}. Actually, our result is slightly stronger
except for $q=\infty$ which is due to the failure of
Littlewood-Paley theory.
\end{remark}
\begin{proof}
From Duhamel's principle, Proposition \ref{p3} by setting $s=s'=0$
and embedding theorem, we immediately get \eqref{eq:Beam2}. Now we
assume $u_0=0$ and $1+\frac{n}{q}-\half n\geq 0$. If $n\geq 2$, then
from Proposition \ref{p3} we have for $|t|\leq 1$,
\begin{eqnarray*}
\normo{\omega^{-1} \sin(t\omega) u_1}_{B_{\infty,2}^0}\les
|t|^{1-\half n}\norm{u_1}_{B_{1,2}^0},
\end{eqnarray*}
and interpolating this with the trivial estimate,
\[
\normo{\omega^{-1} \sin(t\omega) u_1}_2\les |t|\cdot\norm{u_1}_2,
\]
we get that
\[\norm{u(t)}_q\les |t|^{1+\frac{n}{q}-\half n}\norm{u_1}_{q'}.\]
For $n=1$, it suffices to show
\begin{eqnarray*}
\normo{\omega^{-1} \sin(t\omega) u_1}_{B_{\infty,2}^0}\les
|t|^{\half 1}\norm{u_1}_{B_{1,2}^0}.
\end{eqnarray*}
Using Young's inequality, we can easily prove it. We omit the
details.
\end{proof}

3 (Fourth order Schr\"odinger equation). Finally,  we consider the
fourth order Schr\"odinger equation. It is given by
\begin{eqnarray}
i\partial_{t}u+\Delta^2 u-\Delta u=F, \quad u(0)=u_0(x).
\end{eqnarray}
By Duhamel's principle,
\begin{eqnarray*}
u=U(t)u_0-\int_0^tU(t-\tau)F(\tau)d\tau
\end{eqnarray*}
where
\[
U(t)=e^{it(\Delta^2-\Delta)},
\]
which corresponding to $\phi(r)=r^2+r^4$. By simple calculation, we
know $\phi$ satisfies (H1), (H2), (H3) and (H4) with
$m_1=\alpha_1=4$, $m_2=\alpha_2=2$.
\begin{proposition}
Assume $2\leq p\leq \infty$, $\rev p+\rev{p'}=1$, $1\leq q\leq
\infty$, $\delta=\half 1-\rev p$, $-2n\delta\leq s'-s$, then
\begin{eqnarray*}
\norm{U(t)g}_{B_{p,q}^s}\leq k(t)\norm{g}_{B_{p',q}^{s'}}
\end{eqnarray*}
where
\begin{eqnarray*}
k(t)= \left \{
\begin{array}{l}
|t|^{\rev{4}\min(s'-s-2n\delta,\ 0)},\quad |t|\leq 1,\\
|t|^{-n\delta }, \quad |t|\geq1.
\end{array}
\right.
\end{eqnarray*}
\end{proposition}
\begin{proof}
\quad For $|t|\leq 1$, it follows from (c) of theorem \ref{t1} by
setting $\theta=0$ that
\begin{eqnarray*}
\norm{U(t)P_{\leq 0}u_0}_p \lesssim \norm{P_{\leq 0}u_0}_{p'}.
\end{eqnarray*}
For $k>0$, from (a) of Theorem \ref{t1} and interpolation, we get
for $0\leq \theta \leq \half n$,
\begin{eqnarray*}
\norm{U(t)\triangle_ku_0}_p \lesssim
|t|^{-2\theta\delta}2^{2k(n-4\theta)\delta}\norm{\triangle_ku_0}_{p'},
\end{eqnarray*}
and then that
\begin{eqnarray*}
2^{ks}\norm{U(t)\triangle_ku_0}_p \lesssim
|t|^{-2\theta\delta}2^{-k(s'-s-2(n-4\theta)\delta)}2^{ks'}\norm{\triangle_ku_0}_{p'}.
\end{eqnarray*}
If $-2n\delta\leq s'-s \leq 2n\delta$, then we can choose
$0\leq\theta\leq \half n$ such that $s'-s-2n\delta=-8\theta\delta$.
If $s'-s > 2n\delta$, then we choose $\theta=0$. Thus we get
\begin{eqnarray*}
2^{ks}\norm{U(t)\triangle_ku_0}_p \lesssim
|t|^{\rev{4}\min(s'-s-2n\delta,\
0)}2^{ks'}\norm{\triangle_ku_0}_{p'}.
\end{eqnarray*}
Therefore, we get
\begin{eqnarray*}
\norm{U(t)g}_{B_{p,q}^s}\leq |t|^{\rev{4}\min(s'-s-2n\delta,\
0)}\norm{g}_{B_{p',q}^{s'}},
\end{eqnarray*}
which completes the proof in the case $|t|\leq 1$.

For the case $|t|\geq 1$, we can follow the same way as in the proof
of Proposition \ref{p3} to get the result, which completes the
proof.
\end{proof}

\section{Nonlinear Klein-Gordon and Beam equations} \label{NLKG-NLB}

We consider the Cauchy problem for the nonlinear Klein-Gordon
equation (NLKG)
\begin{eqnarray}\label{nlklein}
\partial_{tt}u-\Delta u+u=|u|^{\kappa+1},\ \
u(0)=u_0(x),\ u_t(0)=u_1(x).
\end{eqnarray}
By Duhamel's principle, NLKG is equivalent to
\begin{eqnarray*}
u=K'(t)u_0+K(t)u_1-\int_0^tK(t-\tau)|u(\tau)|^{1+\kappa}d\tau.
\end{eqnarray*}
If $\kappa\ge 4/n$, the global well posedness and the scattering
with small data in $H^s$ were studied in
\cite{Pe1,Pe2,Strauss,Wa4,Wa5,WH}. When Strauss \cite{Strauss}
studied the existence of the scattering operators at low energy, an
important critical power $\kappa(n)$ of the following NLKG
\begin{eqnarray}\label{nlklein-1}
\partial_{tt}u+u-\Delta u+|u|^{\kappa+1}u=0
\end{eqnarray}
was discovered, where
$$
\kappa(n)= \frac{2-n+\sqrt{n^2+12n+4}}{2n}.
$$
Strauss \cite{Strauss} obtained the existence of the scattering
operators at low energy of Eq. \eqref{nlklein-1} in the case
$\kappa(n)<\kappa\le 4/(n-1)$. Since Eq. \eqref{nlklein} has no
conservation of energy, the technique in \cite{Strauss} can not be
directly applied for Eq. \eqref{nlklein}. However, using the basic
decay estimates of the Klein-Gordon equation, we have

\begin{theorem} \label{NLKG}
Let $\kappa(n)<\kappa<4/n$,
$\sigma(\kappa)=\frac{\kappa(n+2)}{2(2+\kappa)}$, $(u_0,u_1)\in
H^{\sigma(\kappa)}_{(2+\kappa)/(1+\kappa)}\times
H^{\sigma(\kappa)-1}_{(2+\kappa)/(1+\kappa)}$ with sufficiently
small norm. Then Eq. \eqref{nlklein} has a unique solution
$$
u\in C( \mathbb{R},\  H^{\kappa/(2+\kappa)}) \cap L^{1+\kappa}
(\mathbb{R}, \ L^{2+\kappa}(\mathbb{R}^n)).
$$
\end{theorem}
\begin{proof}
We present a quite simple proof. Using the basic decay of $K(t)$ and
$K'(t)$, we have
$$
\|K'(t) u_0 \|_{2+\kappa} \lesssim k(t)
\|u_0\|_{H^{\sigma(\kappa)}_{(2+\kappa)/(1+\kappa)}}, \quad \|K(t)
u_1 \|_{2+\kappa} \lesssim k(t)
\|u_1\|_{H^{\sigma(\kappa)-1}_{(2+\kappa)/(1+\kappa)}},
$$
where
$$
k(t)= \left \{\begin{array}{ll}
|t|^{\min(\frac{\kappa(2-n)}{2(2+\kappa)},\
0)},\quad & |t|\leq 1,\\
|t|^{-\frac{\kappa n}{2(2+\kappa)}}, \quad & |t|\geq 1.
\end{array}
\right.
$$
Noticing that if $\kappa(n)<\kappa<4/n$, then we have
$$
(1+\kappa)\frac{\kappa(n-2)}{2(2+\kappa)}<1, \ \
(1+\kappa)\frac{\kappa n}{2(2+\kappa)}>1, \quad \sigma(\kappa)<1.
$$
It follows that $k(\cdot)\in L^{1+\kappa}(\mathbb{R}^n)$ and
$$
\|K'(t) u_0 \|_{L^{1+\kappa}(\mathbb{R}, \ L^{2+\kappa})} \lesssim
\|u_0\|_{H^{\sigma(\kappa)}_{(2+\kappa)/(1+\kappa)}}, \quad \|K(t)
u_1 \|_{L^{1+\kappa}(\mathbb{R}, \ L^{2+\kappa})} \lesssim
\|u_1\|_{H^{\sigma(\kappa)-1}_{(2+\kappa)/(1+\kappa)}}.
$$
In view of Young's and H\"older's inequalities,
\begin{align*}
\normo{\int_0^tK(t-\tau)|u(\tau)|^{1+\kappa}d\tau}_{L^{1+\kappa}(\mathbb{R},
\ L^{2+\kappa})} & \lesssim
\||u|^{\kappa+1}\|_{L^{1+\kappa}(\mathbb{R}, \
H^{\sigma(\kappa)-1}_{(2+\kappa)/(1+\kappa)})} \\
& \lesssim \|u\|^{\kappa+1}_{L^{1+\kappa}(\mathbb{R}, \
L^{2+\kappa})}.
\end{align*}
Taking $M=2C(\|u_0\|_{H^{\sigma(\kappa)}_{(2+\kappa)/(1+\kappa)}}+
\|u_1\|_{H^{\sigma(\kappa)-1}_{(2+\kappa)/(1+\kappa)}})$ and
$$
X=\{ u\in L^{1+\kappa}(\mathbb{R}, \ L^{2+\kappa}):\
\|u\|_{L^{1+\kappa}(\mathbb{R}, \ L^{2+\kappa})} \le M\}.
$$
Observing the mapping
\begin{eqnarray*}
\mathscr{T}: u \to
K'(t)u_0+K(t)u_1-\int_0^tK(t-\tau)|u(\tau)|^{1+\kappa}d\tau,
\end{eqnarray*}
we have
\begin{eqnarray*}
\|\mathscr{T} u\|_X  \le  M/2 + CM^{1+\kappa}.
\end{eqnarray*}
If $C M^\kappa\le 1/2$, we see that $\mathscr{T}: X\to X$ is a
contraction mapping. Hence, Eq. \eqref{nlklein} has a unique
solution $u\in X$.  Moreover,
$$
\|u\|_{L^\infty (\mathbb{R}, \ H^{\kappa/(2+\kappa)})} \lesssim
\|u_0\|_{H^{\kappa/(2+\kappa)}} + \|u_1\|_{H^{-2/(2+\kappa)}} +
\||u|^{\kappa+1}\|_{L^1(\mathbb{R}, H^{-2/(2+\kappa)})}.
$$
Using the embedding $H^{\sigma(\kappa)}_{(2+\kappa)/(1+\kappa)}
\subset H^{\kappa/(2+\kappa)}$, we immediately get that  $u\in
L^\infty (\mathbb{R}, \ H^{\kappa/(2+\kappa)})$.
\end{proof}

Using the method as in the NLKG, we consider the Cauchy problem for
the nonlinear Beam equation (NLB)
\begin{eqnarray}\label{NLB}
\partial_{tt}u+\Delta^2 u+u=|u|^{\kappa+1},\ \
u(0)=u_0(x),\ u_t(0)=u_1(x).
\end{eqnarray}
By Duhamel's principle, NLB is equivalent to
\begin{eqnarray*}
u=B'(t)u_0+B(t)u_1-\int_0^tB(t-\tau)|u(\tau)|^{1+\kappa}d\tau.
\end{eqnarray*}
If $\kappa\ge 8/n$, the global well posedness and scattering with
small data for the NLB were studied in \cite{LEV}, \cite{Wa2}.
Following the same ideas as in the NLKG, we find a critical power
$$
\kappa_B(n)= \frac{4-n+\sqrt{n^2+24n+16}}{2n}.
$$

\begin{theorem}
Let $\kappa_B(n)<\kappa<8/n$,
$\sigma(\kappa)=\frac{n\kappa}{(2+\kappa)}-\frac{2}{1+\kappa}$,
$\sigma(\kappa)<s \leq 2$, $s_2=s-\frac{\kappa n}{2(2+\kappa)}$,
$(u_0,u_1)\in H^{s}_{(2+\kappa)/(1+\kappa)}\times
H^{s-2}_{(2+\kappa)/(1+\kappa)}$ with sufficiently small norm. Then
Eq. \eqref{NLB} has a unique solution
$$
u\in C( \mathbb{R},\  H^{s_2}) \cap L^{1+\kappa} (\mathbb{R}, \
L^{2+\kappa}(\mathbb{R}^n)).
$$
\end{theorem}
\begin{proof}
The proof is similar to that of Theorem \ref{NLKG}. It is easy to
verify that $\sigma(\kappa)\in (0,2)$. Using the basic decay of
$B(t)$ and $B'(t)$, we have
$$
\|B'(t) u_0 \|_{2+\kappa} \lesssim k(t)
\|u_0\|_{H^{s}_{(2+\kappa)/(1+\kappa)}}, \quad \|B(t) u_1
\|_{2+\kappa} \lesssim k(t)
\|u_1\|_{H^{s-2}_{(2+\kappa)/(1+\kappa)}},
$$
where
$$
k(t)= \left \{\begin{array}{ll}
|t|^{\min(\frac{s}{2}-\frac{n\kappa}{2(2+\kappa)},\
0)},\quad & |t|\leq 1,\\
|t|^{-\frac{\kappa n}{4(2+\kappa)}}, \quad & |t|\geq 1.
\end{array}
\right.
$$
Noticing that if $\kappa_B(n)<\kappa<8/n$ and $\sigma(\kappa)<s\leq
2$, then we have
$$
(1+\kappa)\brk{\frac{n\kappa}{2(2+\kappa)}-\frac{s}{2}}<1, \ \
(1+\kappa)\frac{\kappa n}{4(2+\kappa)}>1.
$$
It follows that $k(\cdot)\in L^{1+\kappa}(\mathbb{R}^n)$ and
$$
\|B'(t) u_0 \|_{L^{1+\kappa}(\mathbb{R}, \ L^{2+\kappa})} \lesssim
\|u_0\|_{H^{s}_{(2+\kappa)/(1+\kappa)}}, \quad \|B(t) u_1
\|_{L^{1+\kappa}(\mathbb{R}, \ L^{2+\kappa})} \lesssim
\|u_1\|_{H^{s-2}_{(2+\kappa)/(1+\kappa)}}.
$$
In view of Young's and H\"older's inequalities,
\begin{align*}
\normo{\int_0^t
B(t-\tau)|u(\tau)|^{1+\kappa}d\tau}_{L^{1+\kappa}(\mathbb{R}, \
L^{2+\kappa})} & \lesssim \||u|^{\kappa+1}\|_{L^{1}(\mathbb{R}, \
H^{s-2}_{(2+\kappa)/(1+\kappa)})} \\
& \lesssim \|u\|^{\kappa+1}_{L^{1+\kappa}(\mathbb{R}, \
L^{2+\kappa})}.
\end{align*}
Then, following the same way as in the proof of Theorem \ref{NLKG},
we can prove the result, as desired.
\end{proof}

\noindent{\bf Acknowledgment.}   This work is supported in part by
the NSF of China, grants 10471002, 10571004; RFDP of China, grants
20060001010; and the 973 Project Foundation of China, grant
2006CB805902.

\end{document}